\documentclass{article}
 \usepackage{amsmath, amscd, amssymb, amsthm}
 \usepackage{fullpage}

\usepackage{graphicx}
\newtheorem{theorem}{Theorem}

\newtheorem{lemma}{Lemma}

\newtheorem{prop}{Proposition}

\newtheorem{conj}{Conjecture}

\newcommand{\ignore}[1]{}

\newcommand{\bM}{\mathbf{M}}

\begin{document}

\title{Bounds for the Zero-Forcing Number of Graphs with Large Girth}
\author{Randy Davila\footnote{randy.r.davila@rice.edu; Rice University; Houston, TX, USA}, Franklin Kenter\footnote{Corresponding author: franklin.h.kenter@rice.edu; Rice University; Houston, TX, USA}}
\maketitle

\begin{abstract}
We investigate the zero-forcing number for triangle-free graphs. We improve upon the trivial bound, $\delta \le Z(G)$ where $\delta$ is the minimum degree, in the triangle-free case. In particular, we show that $2 \delta - 2 \le Z(G)$ for graphs with girth of at least 5, and this can be further improved when $G$ has a small cut set.
Lastly, we make a conjecture that the lower bound for $Z(G)$ increases as a function of the girth, $g$, and $\delta$.  
\end{abstract}

\section{Introduction}

The minimum rank of a graph seeks the minimum rank over all matrices whose sparsity pattern is determined by the graph. The zero forcing process, and the associated zero forcing number, were introduced by \cite{AIM-Workshop} and \cite{quantum1} in order to bound the minimum rank of a graph. Further, the zero forcing process and its variants have expanded applications in power network monitoring \cite{powerdom3, powerdom2}, quantum physics \cite{quantum1,quantum2}, and logic \cite{logic1}. Since then, the zero forcing number has gained much attention in graph theory and has been related to  many graph theoretic parameters including minimum degree, d'arboresense, treewidth, pathwidth, among others \cite{zf_tw}. We give a detailed description of the zero forcing process in the next section.

In general, the study of the zero forcing number is challenging for many reasons. First, it is difficult to compute exactly, as it is NP-hard \cite{zf_np}. Further, many of the known bounds leave a wide gap for graphs in general. For example, given the minimum and maximum degree of a graph, $\delta$ and $\Delta$ respectively, the zero forcing number on a graph with $n$ vertices can be as low as $\delta$ \cite{zf_tw} , and as high as $\frac{n \Delta}{\Delta+1}$ \cite{Yair}.

To date, the only result concerning minimum rank or zero forcing and triangle-free graphs is given by Deaett \cite{sdmr_tf} where it is shown that the \emph{semidefinite minimum rank} of a triangle-free graph is bounded below by half the number of vertices.
Our main contribution is the improvement of the lower bound mentioned above in the case of triangle free graphs and graphs with larger girth. 

This paper is organized as follows: In Section \ref{pre} we give the basic definitions and background. Next, we provide our main results with proofs in Section \ref{main}. 
Finally, in Section \ref{conj}, we conjecture improvements to our main results and provide computational evidence.

\section{Preliminaries}\label{pre}

Our study will be restricted to simple graphs with undirected edges and without loops. Let $G=(V,E)$ be a graph with vertex set $V$ and edge set $E$. We will use standard notation: $n=|V|$, $m=|E|$, minimum vertex degree $\delta(G)=\delta$, and maximum vertex degree $\Delta(G)=\Delta$.  The \emph{girth} of a graph $G$, denoted $g$, is the length of its shortest cycle. If $G$ is a tree, by convention, we say that $g=
\infty$. When $g>3$, we call a graph \emph{triangle-free}. The \emph{open neighborhood} of a vertex $v$, denoted $N(v)$, is the set of all vertices adjacent to $v$, excluding $v$ itself. The \emph{closed neighborhood}, $N[v]$, includes $v$. In addition, given $S \subset V$, let $G[S]$ denote the graph induced by $S$. 

We now define the \emph{zero forcing process} as first described by \cite{AIM-Workshop}:
Let $S=S_{0}$ be an initial set of ``colored'' vertices at time $t=0$. All other vertices are said to be ``uncolored''. At each step, a colored vertex $v$ ``forces'' an uncolored neighbor $w$ to become colored if all of the other neighbors of $v$ are colored. At time $t=1$, every initially colored vertex that can force a neighbor will, and the resulting set of forced neighbors and initially colored vertices we denote by $S_{1}$. The notation $S_{t}$ will be used to denote the set of vertices colored at time $t-1$ together with the all possible vertices forced by $S_{t-1}$ at time $t$.
Further, a vertex $v$ is called \emph{active} at time $t$, if it can possibly force at time $t+1$. 
The process of forming a colored set $S_{t}$ from $S_{t-1}$ we call the zero forcing process. If no vertex can be colored at a time step, then the zero forcing process terminates. This process results in a nested sequence of colored vertices $S = S_{0}\subset S_{1}\subset ... \subset S_{t}$. A set $S$ is called a \emph{zero forcing set}, if for some time $t$,  $S_{t}=V$. The \emph{zero forcing number} of a graph, denoted by $Z(G)$, is the cardinality of a smallest zero forcing set, and was shown to be an $NP$-hard invariant in \cite{zf_np}.

A classical lower bound for the zero forcing number is the following:

\begin{prop}[Barioli-Barrett-Fallat-Hall-Hogben-Shader-van der Holst \cite{zf_tw}] \label{trivial}For any graph $G$, 
\[Z(G) \ge \delta. \]
\end{prop} 

In a similar manner, Amos, Caro, Davila, and Pepper recently provided an upper bound regarding on the zero-forcing number:

\begin{prop}[Amos-Caro-Davila-Pepper \cite{Yair}] \label{maxdegreeupper}For any graph $G$ with $\delta \ge 1$, 
\[Z(G) \le \frac{n \Delta}{\Delta +1}. \]
\end{prop}

Further, it is clear that $Z(G)=n-1$ if and only if $G=K_{n}$. This result leads to the following proposition.

\begin{prop}[see for example \cite{row}]\label{zero forcing cuts}
Let $G$ be a non-empty graph with order $n$. Then,
\begin{equation*}
Z(G)\leq n-2,
\end{equation*}
whenever $G\neq K_{n}$.
\end{prop}


One of the primary motivations of zero forcing is the study of minimum rank of a graph. The minimum rank of a graph, denoted $mr(G)$, is the minimum rank over all symmetric $n \times n$ matrices $\bM$ such that $\bM_{ij} = 0$ whenever $i \ne j$ and $\{i,j\}$ is not an edge and $\bM_{ij} \ne 0$ when $\{i,j\}$ is an edge; the diagonal entries may take any value. The important connection between the minimum rank of graph and zero forcing is the following:

\begin{prop}[AIM-Group \cite{zf_tw}]\label{mrz}
For any graph $G$,
\[ Z(G) \ge n- {mr}(G) \]
\end{prop}

%
%

%
 \section{Main Results}\label{main}

In this section, we describe our main results. The overlying concept of these results is that Proposition \ref{trivial} can be improved drastically under mild requirements. 

\begin{lemma}\label{minimumlemma}
Let $G$ be a triangle free graph. Let $S$ be a minimum zero forcing set of $G$. Let $v\in S$ force $w$ at time $t=1$, then $w$ has at least one neighbor not in $S$.
\end{lemma}

\begin{proof}
Let $G$ be a graph with $\delta \geq 3$, and let $S_{0}$ be a minimum zero forcing set of $G$ such that $v$ forces $w$ at time $t=1$. By way of contradiction assume that $N(w)\subset S_{0}$. Since $\delta\geq 3$, we know there exists $z\in N(w)\setminus \{v\}$. Since $G$ is triangle-free, we know $z\notin N(v)$. Starting with a uncolored copy of $G$, define an initial set of colored vertices $S'_{0} = S_{0}\setminus \{z\}$. Since $v$ is not adjacent to $z$, $v$ forces $w$ at time $t=1$ under the new coloring $S'_{0}$. Since we have assumed that $N(w)$ is initially colored in $S_{0}$, we know that $N(w)\setminus \{z\}$ is colored in $S'_{0}$. It follows that at time $t=2$, $w$ will have $d(w)-1$ neighbors colored, and hence will force $z$. So we have shown that $S_{0} \subset S'_{2}$. Since $S_{0}$ was a zero forcing set of $G$, $S'_{0}$ must also be a zero forcing set of $G$, contradicting the minimality of $S_{0}$.
\end{proof}

\begin{theorem}\label{plus one}
Let $G$ be a triangle-free graph with minimum degree $\delta \ge 3$. Then,
\begin{equation*}
\delta +1 \leq Z(G).
\end{equation*}
\end{theorem}

\begin{proof}
Suppose $G$ is triangle-free, and let $S$ be a zero forcing set realizing $Z(G)$. Since $G$ is not a complete graph, nor an empty graph, by Proposition \ref{zero forcing cuts}, we know at least two forcing vertices exist, $v$ and $w$. Note that $v$ and $w$ may force at separate times. Without loss of generality, suppose $v$ forces $v'$, at time $t=1$. Hence, 
$N[v]\setminus\{v'\}\subseteq S$. It suffices to show that there is an initially colored vertex not in $N[v]\setminus\{v'\}$ at time $t=1$.

Since $w$ is assumed to force eventually, we have the following cases:
\begin{enumerate}
\item $w$ forces at time $t=1$.
\item $w$ forces at time $t\geq 2$.
\end{enumerate}

Case 1. Suppose $w$ forces $w'$, at time $t=1$. If $w$ is not in the neighborhood of $v$, then since $w$ was initially colored, we are done.  
Hence, suppose $w$ is in the neighborhood of $v$. Then because $G$ is triangle-free, $w$ cannot be adjacent to any neighbors of $v$, otherwise $G$ would have a triangle. Furthermore, by assumption $\delta \geq 3$, so $w$ must have at least one neighbor other than $v$ and $w'$, $z$ (say). Since $w$ forces $w'$ at time $t=1$, $z$ must be colored at time $t=0$. Since $w$ is adjacent to $v$, $z$ cannot be a neighbor of $v$, otherwise $G$ would contain a triangle. Therefore $z$ is initially colored outside $N[v]\setminus\{w\}$, and we are done. 
Altogether if there are two vertices $v$ and $w$ that force at time $t=1$, the theorem holds. So for the remaining case we may assume that $v$ is the unique forcing vertex at time $t=1$.

Case 2. Suppose that $w$ forces at some time $t\geq 2$. Since forcing steps occur at each time step, there must be a vertex that forces at time $t=2$; without loss of generality, assume that it is $w$. Since $w$ became active at time $t=1$, and since $v$ is assumed to be the unique forcing vertex at time $t=1$, it must be the case that either $v$ forced $w$ or a neighbor of $w$. If $v$ forces $w$, then we are done by Lemma \ref{minimumlemma}. Otherwise $v$ forced a neighbor of $w$, and therefore $w$ cannot be in the neighborhood of $v$ as $G$ would contain a triangle.
\end{proof}

The following lemma is a useful tool for finding large zero forcing sets.

\begin{lemma}\label{setdiff}
Let $S$ be a set of colored vertices of $G$, and let $B\subseteq S$ be a set of vertices whose entire neighborhood is colored. Then $S$ is a zero forcing set of $G$ if and only if $S\setminus B$ is a zero forcing set of $G[V\setminus B]$.
\end{lemma}

\begin{proof}
\emph{Necessity:} Suppose that $S\setminus B$ is a zero forcing set of $G$. Then adding any set of already colored vertices back to $G$ will clearly still be a zero forcing set of $G$.

\emph{Sufficiency:} Suppose that $S\subseteq V$ is a zero forcing set of $G$. Let $B\subseteq S$ be a set of colored vertices with only colored neighbors. Note that every vertex of $B$ is inactive. Furthermore, any currently active vertex will be active in the graph $V\setminus B$, since no vertex in $B$ has any white neighbors, and hence any active vertex in $S$ will have exactly one white neighbor in $V\setminus B$. Forcing chains will hence proceed as if they were in $G$, and must force the rest of the graph. 
\end{proof}

The following theorem improves on proposition 2 for some families of graphs, such as cycles, but in general is non-comparable. Note that this theorem makes sense only if the graph in question contains a cycle. 

\begin{theorem}\label{GirthUpperBound}
Let $G$ be a graph with girth $g$. Then,
\begin{equation*}
Z(G)\leq n-g+2,
\end{equation*}
and this bound is sharp.
\end{theorem}

\begin{proof}
Let $G$ be a graph with girth $g$. Let $C\subseteq V$, be set of vertices realizing the girth of $G$ that induce a cycle. Color the set $V\setminus C$. Next color any two neighbors of $C$, and call this collection of colored vertices $S$. If $B$ is the set of all colored vertices with only colored neighbors, observe that $S\setminus B$ is a zero forcing set of $G[V\setminus B]$. By Lemma \ref{setdiff}, it follows that $S$ is a zero forcing set of $G$. Since a cycle of order $n$ has girth $g=n$, we see that this bound is sharp.
\end{proof}



\begin{lemma}\label{two neighbors}
Let $G$ be a triangle-free graph with minimum degree $\delta\geq 2$. If every minimum zero forcing set requires at least two neighbors to force at time $t=1$, then
\begin{equation*}
2\delta-2\leq Z(G).
\end{equation*}
\end{lemma}

\begin{proof}
Since $\delta =2$ satisfies the Lemma \ref{two neighbors} trivially, assume $G$ is a triangle-free graph with $\delta \geq 3$ such that every zero forcing set requires at least two neighbors to force at time $t=1$. Let $v_{0}$ and $w_{0}$ be vertices of a minimum zero forcing set $S$ which are neighbors and both force at time $t=1$. By Lemma 2, if $v_{0}$ has $\delta$ neighbors in $S$, then $w_{0}$ has at least $\delta -2$ neighbors in $S_{0}$. Hence $2\delta -2\leq |S|=Z(G)$.
\end{proof}

For graphs with girth greater than 4, we are able to show that the minimum degree lower bound can be improved by a factor of almost 2.

\begin{theorem}\label{girth5}
Let $G$ be a graph with girth $g\geq 5$, and minimum degree $\delta\geq 2$. Then, 
\begin{equation*}
2\delta - 2\leq Z(G).
\end{equation*}
\end{theorem}

\begin{proof}
Notice that when $\delta =2$, the theorem holds. So we consider graphs with $\delta \geq 3$.

Let $G$ be a graph with girth $g\geq 5$, and minimum degree $\delta\geq 3$. Let $S$ be a zero forcing set realizing $Z(G)$. Since $G$ is not a complete graph, nor an empty graph, we know at least two forcing vertices exist, $v$ and $w$ (say). Note that $v$ and $w$ may force at separate times. Without loss of generality, suppose $v$ forces $v'$ (say), at time $t=1$. Hence, 
$N[v]\setminus\{v'\}\subseteq S$. 

Since $w$ is assumed to force eventually, we have the following cases:
\begin{enumerate}
\item $w$ forces at time $t=1$.
\item $w$ forces at time $t\geq 2$.
\end{enumerate}

Case 1. Suppose $w$ forces $w'$ (say), at time $t=1$. Note $w$ is either in the neighborhood of $v$, or it isn't. If $w$ is not in the neighborhood of $v$, then because $w$ and all neighbors of $w$ other then $w'$ were colored at time $t=0$, we have
\begin{equation*}
|S|\geq \delta +\delta -k,
\end{equation*}
Where $k$ is the (possibly zero) number of vertices in 
\begin{equation*}
\big(N[v]\setminus\{v'\}\big)\cap\big(N[w]\setminus\{w'\}\big).
\end{equation*}
Because $g\geq 5$, $w$ cannot be adjacent to more than one neighbor of $v$, since otherwise $v$, $w$, and the two shared neighbors would induce a 4-cycle. Hence, 
\begin{equation*}
\begin{split}
|S|
&\geq \delta + \delta - 1\\
&= 2\delta -1\\
&> 2\delta -2,
\end{split}
\end{equation*}
and the theorem holds.

Next suppose $w$ is in the neighborhood of $v$. Then, because $G$ is triangle-free with $\delta\geq 3$, we recall lemma \ref{two neighbors}, and get,
\begin{equation*}
|S|\geq 2\delta -2.
\end{equation*}


So if $w$ forces at time $t=1$, the theorem holds. Since $w$ was arbitrary, if any vertex other than $v$ forces at time $t=1$, then the theorem holds. Hence, for Case 2 we may assume that $v$ is the unique forcing vertex at time $t=1$.

Case 2. Suppose that $w$ forces at some time $t\geq 2$. Since forcing steps occur at each time step, there must be a vertex that forces at time $t=2$, without loss of generality we take $w$ to be this vertex. Since $w$ became active at time $t=1$, and since $v$ is assumed to be the unique forcing vertex at time $t=1$, it must be the case that $v$ forced a neighbor of $w$, i.e., $v'$ is a neighbor of $w$. Because $v$ forced a neighbor of $w$, $w$ cannot be in the neighborhood of $v$, since otherwise $v$,$v'$, and $w$ would induce a triangle. So it must be the case that $w$ was initially colored. Furthermore, since $w$ is a forcing vertex at time $t=2$, it must have at least $\delta -1$ neighbors that are colored at time $t=1$. Since $v$ was the unique forcing vertex at time $t=1$, the neighbors of $w$ colored at time $t=1$, must have been initially colored at time $t=0$. Since $g\geq 5$, $v$ and $w$ cannot share neighbors, since otherwise the neighbors $v$ and $w$ share, along
  with $v$, $v'$, and $w$, would induce a 4-cycle. Hence,
\begin{equation*}
\begin{split}
|S|
&\geq \big|N[v]\setminus\{v'\}\cup N[w]\setminus\{w'\}\big|\\
& = \big|N[v]\setminus\{v'\}\big|+\big|N[w]\setminus\{w'\}\big|\\
& \geq \delta + \delta \\
&=2\delta\\
&> 2\delta -2,
\end{split}
\end{equation*}
and the theorem holds.
\end{proof}

Before moving to the proofs of Theorems \ref{CutEdge} and \ref{CutVertex}, we make note of a theorem attributed to Edholm Hogben Hyunh LaGrange and Row:

\begin{theorem}[Edholm-Hogben-Hyunh-LaGrange-Row \cite{Edholm}]\label{cut vertex cut edge}
Let $G$ be a graph on $n\geq 2$ vertices, then:

	\begin{enumerate}
	
		\item For $v\in V(G)$, $Z(G)-1\leq Z(G-v)\leq Z(G)+1$.
		
		\item For $e\in E(G)$, $Z(G)-1\leq Z(G-e)\leq Z(G)+1$.
	
	\end{enumerate}
	
\end{theorem}

First we consider graphs with a cut vertex, such that after removal of such a vertex, the resulting graph has at least one component with girth at least 5. We improve on Proposition \ref{trivial} by close to a factor of 3 for such graphs.

\begin{theorem}\label{CutVertex}
Let $G$ be a graph with minimum degree at least 3 and a cut vertex $v$ such that $G-{v}$ has a component with girth at least 5. Then,
\begin{equation*}
3\delta -6\leq Z(G).
\end{equation*}
\end{theorem}

\begin{proof}
Let $G$ be a graph with minimum degree $\delta \geq 3$ and a cut vertex $v$ such that $G-{v}$ has a component with girth at least 5. Let $G-{v}$ have components $H_{1}$ and $H_{2}$. Without loss of generality suppose $H_{1}$ has girth at least 5. Recall $Z(G-{v})\leq Z(G)+1$ by Theorem \ref{cut vertex cut edge}. Since zero forcing is additive with respect to disjoint components we have
\begin{equation*}
Z(H_{1})+Z(H_{2})\leq Z(G)+1.
\end{equation*}
Further, note that $\delta(H_{1},H_{2}) \leq \delta -1$, since deleting $v$ from $G$ at most reduced the degree of the components by 1. By Theorem \ref{girth5}, we know that 
\begin{equation*}
2(\delta -1)-2\leq Z(H_{1}).
\end{equation*}
By the minimum degree lower bound we know $\delta - 1 \leq Z(H_{2})$, and hence
\begin{equation*}
\begin{split}
(2(\delta -1)-2) + (\delta -1)
&= 3\delta -5\\ 
&\leq Z(H_{1})+Z(H_{2})\\
&\leq Z(G)+1.
\end{split}
\end{equation*} 
Rearranging terms, we get our desired result $3\delta -6\leq Z(G)$.
\end{proof}

For graphs with a cut edge and girth at least 5, we can drastically improve the minimum degree lower bound, as illustrated by the following theorem.

\begin{theorem}\label{CutEdge}
Let $G$ be a graph with minimum degree $\delta\geq 3$, girth $g\geq 5$, and cut edge $e$. Then,
\begin{equation*}
4\delta-9\leq Z(G).
\end{equation*}
\end{theorem}

\begin{proof}
Let $G$ be a graph with minimum degree $\delta \geq 3$, girth $g\geq 5$, and cut edge $e$ such that $G-\{e\}= H_{1}\cup H_{2}$. Since no cut edge of $G$ may lie on a cycle, we know that each component $H_{1}$ and $H_{2}$ has girth at least 5. Recall $Z(G-{e})\leq Z(G)+1$, from Theorem \ref{cut vertex cut edge}. Since zero forcing is additive with respect to disjoint components we have 
\begin{equation*}
Z(H_{1})+Z(H_{2})\leq Z(G)+1.
\end{equation*}
Note that $\delta(H_{1},H_{2}) \geq\delta -1$ since deleting $e$ from $G$ at most reduced the degree of the components by 1. By Theorem \ref{girth5}, we know that,
\begin{equation*}
2(\delta -1)-2\leq Z(H_{i}),\:\: i=1,2.
\end{equation*}
Hence, 
\begin{equation*}
\begin{split}
(2(\delta -1)-2) + (2(\delta -1)-2)
&=4\delta -8\\
&\leq Z(H_{1})+Z(H_{2})\\
&\leq Z(G)+1.
\end{split}
\end{equation*}
 Rearranging, we get our desired result, $4\delta -9\leq Z(G)$.
\end{proof}

The following theorem generalizes the previous two theorems to include $k$-connected graphs:

\begin{theorem}\label{sequencedelete}
Let $G$ be a graph and suppose there exists a set of vertices $K=\{v_{1}, ...,v_{k}\}$ such that $\delta (G\setminus K)\geq 2$ and $G\setminus K$ has induced girth at least 5. Then,
\begin{equation*}
2\delta - 3k -2 \leq Z(G).
\end{equation*}
\end{theorem}

\begin{proof}
Let $K$ be a set of $k$ vertices whose removal from $G$ forms a triangle-free $C_{4}$-free graph with minimum degree at least 2. Next observe that $Z(G\setminus K)\leq Z(G)+k$. Since removing a vertex will most decrease the minimum degree of $G$ by one, removing $k$ vertices will at most reduce the minimum degree of $G$ by $k$. By Theorem 2, we have $2\delta(G\setminus K) -2 \leq Z(G\setminus K)\leq Z(G)+k$. Rearranging terms we get $2\delta - 3k - 2\leq Z(G)$. 
\end{proof}

\section{A General Conjecture}\label{conj}

In this section, we make a conjecture regarding improvements to Theorems 1 and \ref{girth5}. We provide intuitive, numerical, and theoretical justification for these conjectures. 
We conjecture a lower bound for $Z(G)$ that is dependent on both the girth $g$ and minimum degree $\delta$:

\begin{conj}\label{girthconj}
Let $G$ be a graph with girth $g \ge 3$ and minimum degree $\delta \geq 2$. Then,
\begin{equation*}
(g-3)(\delta - 2) + \delta \le Z(G).
\end{equation*}
\end{conj}

The intuitive reasoning behind this conjecture is the following. Let $G$ be a graph with girth $g$ and minimum degree $\delta$. Consider any path $v_0 \to v_1 \to \ldots \to v_\ell$ such that $v_i$ forces $v_{i+1}$ for all $i$ and $v_0$ forces $v_1$ at $t=0$. It must be the case that, for all $0<j<g$ and $0 \le k \le g$, $v_j$ is not adjacent to $v_k$. Hence, $v_j$ requires at least $\delta - 2$ neighbors outside of the path to be colored. Under the assumption that each of these vertices requires a distinct initially colored vertex in order to become colored, each $v_1, v_2, \ldots, v_{g-3}$  require a total of $(\delta-2)(g-3)$ initially colored vertices. Note that beginning with $j=g-2$, the neighbors of $v_j$ may, in fact be the neighbors of $v_0$ without violating the girth condition. Hence, together with the initial $\delta$ vertices needed for $v_0$ to force $v_1$ requires at least $\delta+(\delta-2)(g-3)$ vertices to be initially colored.

In addition, we can prove the following:

\begin{prop}\label{partialconj}
Given girth $g>6$, Conjecture \ref{girthconj} is true for sufficiently large minimum degree $\delta$, regardless of  $n$.
\end{prop}

To prove this proposition, we use the \emph{tree-width} parameter, denoted ${tw}(G)$, as an intermediary. Tree-width is a graph parameter regarding the decomposition of the graph into sets, however, the details are not important for our application. For more details, refer to \cite{West}.

We apply two results together to prove Proposition \ref{partialconj}: an exponential lower bound on tree-width in terms of $\delta$ and $g$, and a lower bound on $Z(G)$ in terms of tree-width:

\begin{theorem}[Chandrana and Subramanian \cite{girth_tw}]
Let $G$ be a graph with average degree $\bar d$, girth $g$, and tree-width $tw(G)$. Then,
\[ tw(G) \ge
\frac
{(\bar d-1)^{\lfloor(g-1)/2\rfloor -1}}
{12(g + 1)}
\]
\end{theorem}

\begin{theorem}[Barioli-Barrett-Fallat-Hall-Hogben-Shader-van der Holst \cite{zf_tw}]
\[ tw(G) \le Z(G) \]
\end{theorem}

\begin{proof}[Proof of Proposition \ref{partialconj}]
Given $g$, choose $\delta_{min}$ large enough such that 
\[ \frac{( \delta_{min}-1)^{\lfloor(g-1)/2\rfloor -1}}
{12(g + 1)} \ge \delta_{min}+(\delta_{min}-2)(g-3). \]
This is guaranteed for $g>6$ as the left side has polynomial degree at least 2 in $\delta$.
In which case, by the previous two theorems, we have 
\begin{eqnarray*}
\delta_{min}+(\delta_{min}-2)(g-3)  &\le& \frac{(\delta_{min}-1)^{\lfloor(g-1)/2\rfloor -1}}{12(g + 1)} \\ 
&\le& \frac{(\bar d-1)^{\lfloor(g-1)/2\rfloor -1}}{12(g + 1)} \\
&\le& tw(G) \\
&\le& Z(G).
\end{eqnarray*}
Hence, given $g$, the conjecture holds true for all graphs with $\delta > \delta_{min}$.
\end{proof}

We are able to numerically verify our conjecture for a large number of graphs. However, considering that calculating the zero forcing number of a graph is NP-hard \cite{zf_np}, it is difficult to test. We write a compact algorithm in \emph{Mathematica} in order to compute $Z(G)$ using a brute-force approach. In turn, we are able to confirm Conjecture \ref{girthconj}  for  all 542 triangle-free graphs in the {\it Wolfram} database with between 9 and 22 vertices. In fact, our computation shows that for more than 37\% of these graphs, the conjecture is sharp. Conjecture \ref{girthconj} appears to be sharp for several interesting graphs including the Petersen Graph, the Heawood Graph and several families of graphs as well including cycles, $2 \times 2 \times k$ grids, complete bipartite graphs, among others. 

Furthermore, taking $g=4$ in Conjecture \ref{girthconj}, yields a conjecture this is similar in nature to the results presented in Section \ref{main}:

\begin{conj}\label{trifreeconj}
Let $G$ be a triangle-free graph with minimum degree $\delta \geq 2$. Then, 
\begin{equation*}
2\delta -2 \leq Z(G).
\end{equation*}
\end{conj}

\section{Concluding Remarks and Acknowledgements}
The results presented in this paper show a link between the girth and minimum zero forcing sets. In particular, the theorems presented provide evidence in favor to Conjectures \ref{girthconj} and \ref{trifreeconj}. Further, the connection between minimum degree and minimum zero forcing sets remain, since the minimum degree appears in all of presented theorems. Future work will include working toward resolving Conjectures \ref{girthconj} and \ref{trifreeconj}, but also will include finding new lower bound on the zero forcing number in terms of simple graph parameters such as degree, size, and order.

This research was partially supported by NSF CMMI-1300477 and CMMI-1404864.

\bibliographystyle{plain}
\bibliography{bibTrangleFree}

\begin{thebibliography}{10}

\bibitem{AIM-Workshop}
{AIM Special Work Group}.
\newblock Zero forcing sets and the minimum rank of graphs.
\newblock {\em Linear Algebra and its Applications}, 428(7):1628--1648, 2008.

\bibitem{Yair}
David Amos, Yair Caro, Randy Davila, and Ryan Pepper.
\newblock Upper bounds on the $k$-forcing number of a graph.
\newblock {\em arXiv preprint arXiv:1401.6206}, 2014.

\bibitem{zf_tw}
Francesco Barioli, Wayne Barrett, Shaun~M. Fallat, H.~Tracy Hall, Leslie
  Hogben, Bryan Shader, Pauline van~den Driessche, and Hein van~der Holst.
\newblock Parameters related to tree-width, zero forcing, and maximum nullity
  of a graph.
\newblock {\em Journal of Graph Theory}, 72(2):146--177, 2013.

\bibitem{quantum1}
Daniel Burgarth and Vittorio Giovannetti.
\newblock Full control by locally induced relaxation.
\newblock {\em Physical review letters}, 99(10):100501, 2007.

\bibitem{logic1}
Daniel Burgarth, Vittorio Giovannetti, Leslie Hogben, Simone Severini, and
  Michael Young.
\newblock Logic circuits from zero forcing.
\newblock {\em arXiv preprint arXiv:1106.4403}, 2011.

\bibitem{girth_tw}
L~Sunil Chandran and CR~Subramanian.
\newblock Girth and treewidth.
\newblock {\em Journal of combinatorial theory, Series B}, 93(1):23--32, 2005.

\bibitem{sdmr_tf}
Louis Deaett.
\newblock The minimum semidefinite rank of a triangle-free graph.
\newblock {\em Linear Algebra and its Applications}, 434(8):1945--1955, 2011.

\bibitem{Edholm}
Christina~J. Edholm, Leslie Hogben, Joshua LaGrange, and Darren~D. Row.
\newblock Vertex and edge spread of zero forcing number, maximum nullity, and
  minimum rank of a graph.
\newblock {\em Linear Algebra and its Applications}, 436(12):4352--4372, 2012.

\bibitem{powerdom3}
Teresa~W. Haynes, Sandra~M. Hedetniemi, Stephen~T. Hedetniemi, and Michael~A.
  Henning.
\newblock Domination in graphs applied to electric power networks.
\newblock {\em SIAM Journal on Discrete Mathematics}, 15(4):519--529, 2002.

\bibitem{quantum2}
L.~Hogben, D.~Burgarth, Domenico D'Alessandro, Simone Severini, and Michael
  Young.
\newblock Zero forcing, linear and quantum controllability for systems evolving
  on networks.
\newblock 2011.

\bibitem{row}
Darren~Daniel Row.
\newblock {\em Zero forcing number: Results for computation and comparison with
  other graph parameters}.
\newblock PhD thesis, Iowa State University, 2011.

\bibitem{zf_np}
Maguy Trefois and Jean-Charles Delvenne.
\newblock Zero forcing sets, constrained matchings and minimum rank.
\newblock {\em Linear and Multilinear Algebra. Submitted for publication},
  2013.

\bibitem{West}
Douglas~Brent West.
\newblock {\em Introduction to graph theory}, volume~2.
\newblock Prentice hall Upper Saddle River, 2001.

\bibitem{powerdom2}
Min Zhao, Liying Kang, and Gerard~J. Chang.
\newblock Power domination in graphs.
\newblock {\em Discrete mathematics}, 306(15):1812--1816, 2006.

\end{thebibliography}

\end{document}